\documentclass[a4paper,12pt]{amsart}
\usepackage[english]{babel}
\usepackage{amsfonts, amsmath}
\usepackage{xcolor}
\usepackage{amssymb}
\usepackage{latexsym}
\usepackage{exscale}
\usepackage{amssymb}
\usepackage{amsfonts}
\usepackage{color}

\newtheorem{thm}{Theorem}[section]
\newtheorem{prop}[thm]{Proposition}
\newtheorem{lem}[thm]{Lemma}
\newtheorem{cor}[thm]{Corollary}

\newtheorem{rems}[thm]{Remarks}
\newtheorem{defi}[thm]{Definition}
\newtheorem{exo}{\bf\large Exercice}

\newcommand{\R}{\mathbb{R}}

\newcommand{\N}{\mathbb{N}}

\newcommand{\dint}{\displaystyle\int}

\newcommand{\Sum}{\displaystyle \sum}

\newcommand{\Int}{\displaystyle \int}
\newcommand{\Frac}{\displaystyle \frac}

\newcommand{\Sup}{\displaystyle \sup}
\newcommand{\Lim}{\displaystyle \lim}
\newcommand{\Liminf}{\displaystyle \liminf}
\newcommand{\Limsup}{\displaystyle \limsup}
\newcommand{\Max}{\displaystyle \max}

\newcommand{\beq}{\begin{eqnarray}}
\newcommand{\eeq}{\end{eqnarray}}
\newcommand{\bq}{\begin{equation}}
\newcommand{\eq}{\end{equation}}
\newcommand{\beqn}{\begin{eqnarray*}}
\newcommand{\eeqn}{\end{eqnarray*}}
\newcommand{\bex}{\begin{exo}}
\newcommand{\eex}{\end{exo}}
\newcommand{\ben}{\begin{enumerate}}
\newcommand{\een}{\end{enumerate}}

\let\de=\delta

\def\cC{{\mathcal C}}
\def\cD{{\mathcal D}}

\def\cP{{\mathcal P}}

\def\eqdefa{\buildrel\hbox{\footnotesize def}\over =}

\author{Ines Ben Ayed}
\address{Universit\'e de Tunis El Manar, Facult\'e des Sciences de Tunis, LR03ES04 \'Equations aux d\'eriv\'ees partielles et applications, 2092 Tunis, Tunisie}
\email{\sl abenyed08@gmail.com}

\author{Mohamed Khalil Zghal}
\address{Universit\'e de Tunis El Manar, Facult\'e des Sciences de Tunis, LR03ES04 \'Equations aux d\'eriv\'ees partielles et applications, 2092 Tunis, Tunisie}
\email{\sl zghal-khalil@hotmail.fr}

\title[Description of the lack of compactness ...]
{Description of the lack of compactness in Orlicz spaces and
applications}
\date{\today}
\begin{document}

\begin{abstract}
In this paper, we investigate the lack of compactness of the Sobolev
embedding of $H^1(\R^2)$ into the Orlicz space $L^{{\phi}_p}(\R^2)$
associated to the function $\phi_p$ defined by
$\phi_p(s):={\rm{e}^{s^2}}-\Sum_{k=0}^{p-1} \frac{s^{2k}}{k!}\cdot$
We also undertake the study of a nonlinear wave equation with
exponential growth  where the Orlicz norm $\|.\|_{L^{\phi_p}}$ plays
a crucial role. This study includes issues of global existence,
scattering and qualitative study.
\end{abstract}



\maketitle



\section{Introduction}


\subsection{Critical $2D$ Sobolev embedding}
It is well known (see for instance \cite{BCD}) that $H^1(\R^2)$ is
continuously embedded in all Lebesgue spaces $L^q(\R^2)$ for $2\leq
q<\infty$, but not in $L^{\infty}(\R^2 )$. It is also known that
(for more details, we refer the reader to \cite{Orlicz-Book})
\begin{equation}\label{embedding}
H^1(\R^2) \hookrightarrow L^{{\phi}_p}(\R^2),\quad \forall p\in
\N^*,
\end{equation}
where $L^{{\phi}_p}(\R^2)$ denotes the Orlicz space associated to
the function
\begin{equation}\label{phi}
    \displaystyle \phi_p(s)=\rm{e}^{s^2}-\sum_{k=0}^{p-1} \frac{s^{2k}}{k!}\, \cdot
\end{equation}
 The embedding \eqref{embedding} is a direct consequence of the following sharp Trudinger-Moser type inequalities (see \cite {AT, M, Ruf, Tru}):
\begin{prop}
\begin{equation}\label{Mos1}
 \sup_{\|u\|_{H^1}\leq 1}\;\;\Int_{\R^2}\,\left({\rm
e}^{4\pi |u|^2}-1\right)\,dx:=\kappa<\infty,
\end{equation}
\end{prop}
\noindent and states as follows:
\begin{equation}
\label{2D-embed}
\|u\|_{L^{\phi_p}}\leq\frac{1}{\sqrt{4\pi}}\|u\|_{H^1},
\end{equation}
where the norm $\|.\|_{L^{{\phi}_p}}$ is given by:
\begin{equation*}
\|u\|_{L^{\phi_p}}=\inf\,\left\{\,\lambda>0,\int_{\R^d}\,\phi_p\left(\frac{|u(x)|}{\lambda}\right)\;dx\leq
\kappa\,\right\}.
\end{equation*}
For our purpose, we shall resort to the following Trudinger-Moser
inequality, the proof of which is postponed in the appendix.
\begin{prop}\label{Mos3}
Let $\alpha\in[0,4\pi[$ and $p$ an integer larger than $1$. There is
a constant $c(\alpha,p)$ such that
\begin{equation}\label{Trudinger-Moser}
\dint_{\R^2}\left({\rm
e}^{\alpha|u(x)|^2}-\Sum_{k=0}^{p-1}\frac{\alpha^k|u(x)|^{2k}}{k!}\right)\,dx\leq
c(\alpha,p)\|u\|_{L^{2p}(\R^2)}^{2p},
\end{equation}
for all $u\in H^1(\R^2)$ satisfying $\|\nabla u\|_{L^2(\R^2)}\leq
1$.
\end{prop}

\subsection{Development on the lack of compactness of Sobolev embedding in the Orlicz space in  the case $p=1$}
In \cite{Bahouri1}, \cite{Bahouri2} and \cite{Bahouri}, H. Bahouri,
M. Majdoub and N. Masmoudi characterized the lack of compactness of
$H^1(\R^2)$ into the Orlicz space $L^{\phi_1}(\R^2)$. To state their
result in a clear way, let us recall some definitions.

\begin{defi}
\label{ortho} We shall designate by a scale any sequence
$(\alpha_n)$ of positive real numbers going to infinity, a core any
sequence $(x_n)$ of points in $\R^2$ and a profile any function
$\psi$ belonging to the set
$$
{\cP}:=\Big\{\;\psi\in L^2(\R,{\rm e}^{-2s}ds);\;\;\; \psi'\in
L^2(\R),\;\psi_{|]-\infty,0]}=0\,\Big\}.
$$
Given two scales $(\alpha_n)$, $(\tilde{\alpha}_n)$, two cores
$(x_n)$, $(\tilde{x}_n)$ and tow profiles $\psi$, $\tilde{\psi}$, we
say that the triplets $\big((\alpha_n),(x_n),\psi\big)$ and
$\big((\tilde{\alpha}_n),(\tilde{x}_n),\tilde{\psi}\big)$ are
orthogonal if
    $$\mbox{either}\quad\quad
   \Big|\log\left(\tilde{\alpha}_n/{\alpha}_n\right)\Big|\to\infty,
    $$
or $  \tilde \alpha_n =  \alpha_n $ and
$$ - \frac{ \log|x_n- \tilde x_n|}{\alpha_n} \longrightarrow a \geq 0  \,\, \mbox{with} \,\,\psi \,\, \mbox{or} \,\, {\tilde \psi}\,\, \mbox{null for} \,\, s < a\,.$$
\end{defi}
\begin{rems}\quad\\
\vspace{-0.5cm}
\begin{itemize}
\item
 The profiles belong to the H\"older space $C^\frac{1}{2}$. Indeed, for any profile $\psi$ and real numbers $s$ and $t$, we have by Cauchy-Schwarz inequality
$$|\psi(s)-\psi(t)|=\left|\int_s^t\psi'(\tau)\;d\tau\right|\leq\|\psi'\|_{L^2(\R)}|s-t|^\frac{1}{2}.$$
\item Note also that (see \cite{Bahouri1})
\begin{equation}\label{prop}
    \frac{\psi(s)}{\sqrt{s}}\rightarrow 0 \quad as \quad s\rightarrow 0 \quad and \quad as \quad s\rightarrow \infty.
\end{equation}
\end{itemize}
\end{rems}
The asymptotically orthogonal decomposition derived in
\cite{Bahouri2} is  formulated in the following terms:
 \begin{thm}
\label{noradmain} Let $(u_n)$ be a bounded sequence in $H^1(\R^2)$
such that \bq \label{noradmain-assum1} u_n\rightharpoonup 0, \eq \bq
\label{noradmain-assum2}
\limsup_{n\to\infty}\|u_n\|_{L^{\phi_1}}=A_0 >0 \quad \quad
\mbox{and}\eq \bq \label{noradmain-assum4} \lim_{R\to\infty}\;
\limsup_{n\to\infty}\,\|u_n\|_{L^{\phi_1} (|x|>R)}=0. \eq Then,
there exist a sequence of scales $({\alpha}_n^{(j)})$, a  sequence
of cores $({x}_n^{(j)})$ and a sequence of profiles $(\psi^{(j)})$
such that the triplets $({\alpha}_n^{(j)}, {x}_n^{(j)},\psi^{(j)})$
are pairwise orthogonal and, up to a subsequence extraction, we have
for all $\ell\geq 1$, \bq \label{noraddecomp}
u_n(x)=\Sum_{j=1}^{\ell}\,\sqrt{\frac{\alpha_n^{(j)}}{2\pi}}\;\psi^{(j)}\left(\frac{-\log|x
- x_n^{(j)}|}{\alpha_n^{(j)}}\right)+{\rm
r}_n^{(\ell)}(x),\quad\limsup_{n\to\infty}\;\|{\rm
r}_n^{(\ell)}\|_{L^{\phi_1}}\stackrel{\ell\to\infty}\longrightarrow
0. \eq Moreover, we have the following stability estimate \bq
\label{ortogonal} \|\nabla
u_n\|_{L^2}^2=\Sum_{j=1}^{\ell}\,\|{\psi^{(j)}}'\|_{L^2}^2+\|\nabla
r_n^{(\ell)}\|_{L^2}^2+\circ(1),\quad n\to\infty. \eq
\end{thm}
\begin{rems}\quad\\
\vspace{-0.5cm}
\begin{itemize}
\item It will be useful later on to point out that for any $q\geq 2$, we have
\begin{equation}\label{lim gnj}
\|g_n^{(j)}\|_{L^q}\stackrel{n\rightarrow\infty}\longrightarrow 0,
\end{equation}
where $g_n^{(j)}$ is the elementary concentration involving in
Decomposition \eqref{noraddecomp} defined by
\begin{equation}\label{def gnj} g_n^{(j)}(x) :=
\sqrt{\frac{\alpha_n^{(j)}}{2\pi}}\;\psi^{(j)}\left(\frac{-\log|x -
x_n^{(j)}|}{\alpha_n^{(j)}}\right).
\end{equation}
Since the Lebesgue measure is invariant under translations, we have
\begin{equation*}
\|g_n^{(j)}\|_{L^q}^q=(2\pi)^{-\frac{q}{2}}(\alpha_n^{(j)})^{\frac{q}{2}}\int_{\R^2}
\bigg|\psi^{(j)}\bigg(-\frac{\log|x|}{\alpha_n^{(j)}}\bigg)\bigg|^{q}
dx.
\end{equation*}
Performing the change of variable
$s=-\frac{\log|x|}{\alpha_n^{(j)}}$, yields
\begin{equation*}
\|g_n^{(j)}\|_{L^q}^q=(2\pi)^{1-\frac{q}{2}}(\alpha_n^{(j)})^{\frac{q}{2}+1}\int^{\infty}_{0}
\big|\psi^{(j)}(s)\big|^{q} {\rm e}^{-2\alpha_n^{(j)}s}\; ds.
\end{equation*}
Fix $\varepsilon > 0$. Then in view of \eqref{prop}, there exist two
real numbers $s_0$ and $S_0$ such that $0<s_0<S_0$ and
\begin{equation*}
    \left|\psi^{(j)}(s)\right|\leq \varepsilon \sqrt{s}, \quad\forall\, s\in [0,s_0]\cup [S_0,\infty[.
\end{equation*}
This implies, by the change of variable $u=\alpha_n^{(j)}s$, that
\begin{eqnarray*}
   (\alpha_n^{(j)})^{\frac{q}{2}+1}\int_0^{s_0} \left|\psi^{(j)}(s)\right|^{q} {\rm e}^{-2\alpha_n^{(j)}s}\; ds & \leq &{\varepsilon}^q  \int_{0}^{\alpha_n^{(j)}s_0} u^{\frac{q}{2}} {\rm e}^{-2u}\;du\\
 & \leq & C_q\,\varepsilon^q.
\end{eqnarray*}
In the same way, we obtain
\begin{eqnarray*}
(\alpha_n^{(j)})^{\frac{q}{2}+1}\int_{S_0}^\infty
\left|\psi^{(j)}(s)\right|^{q} {\rm e}^{-2\alpha_n^{(j)}s}\; ds   &
\leq & C_q\,\varepsilon^q.
\end{eqnarray*}
Finally  taking advantage of  the continuity of $\psi^{(j)}$, we
deduce that \begin{eqnarray*}
    (\alpha_n^{(j)})^{\frac{q}{2}+1}\int_{s_0}^{S_0} \left|\psi^{(j)}(s)\right|^{q} {\rm e}^{-2\alpha_n^{(j)}s}\; ds&\lesssim&  (\alpha_n^{(j)})^{\frac{q}{2}+1}\int_{s_0}^{S_0} {\rm e}^{-2\alpha_n^{(j)}s} \;ds\\
&\lesssim&(\alpha_n^{(j)})^{\frac{q}{2}}\left({\rm
e}^{-2\alpha_n^{(j)}s_0}-{\rm
e}^{-2\alpha_n^{(j)}S_0}\right)\stackrel{n\rightarrow\infty}{\longrightarrow}0,
\end{eqnarray*}
which ends the proof of  the assertion \eqref{lim gnj}.
\item Recall that it was proved in \cite{Bahouri} that
\begin{equation*} \|g_n^{(j)}\|_{L^{\phi_1}}\stackrel{n\rightarrow\infty}\longrightarrow
 \frac{1}{\sqrt{4\pi}}\,\max_{s>0}\;\frac{|\psi^{(j)}(s)|}{\sqrt{s}}\,
\end{equation*}
and
\begin{equation}\label{OrliczMax1} \big\|\Sum_{j=1}^{\ell}\,g_n^{(j)}\big\|_{L^{\phi_1}}\stackrel{n\to\infty}\longrightarrow\sup_{1\leq
j\leq\ell}\,\left(\lim_{n\to\infty}\,\|g_n^{(j)}\|_{L^{\phi_1}}\right)\,
,
\end{equation}in the case when the scales $(\alpha_n^{(j)})_{1 \leq j \leq \ell}$ are pairwise orthogonal.
 Note that Property \eqref{OrliczMax1} does not necessarily remain true in the case when we have the same scales and the pairwise orthogonality of the couples $\big((x^{(j)}_n),\psi^{(j)}\big)$ (see Lemma $3.6$ in \cite{Bahouri}).

\end{itemize}
\end{rems}
\subsection{Study of  the lack of compactness of Sobolev embedding in the Orlicz space in the case $p>1$}
Our first goal  in this paper  is to describe  the lack of
compactness of the Sobolev embedding \eqref{embedding} for $p > 1$.
Our result states as follows:
\begin{thm}
\label{theorem} Let $p>1$ be an integer larger than $1$ and $(u_n)$
be a bounded sequence in $H^1(\R^2)$ such that \bq \label{assum1}
u_n\rightharpoonup 0, \eq \bq \label{assum2}
\limsup_{n\to\infty}\|u_n\|_{L^{\phi_p}}=A_0 >0 \quad \quad
\mbox{and}\eq \bq \label{assum3} \lim_{R\to\infty}\;
\limsup_{n\to\infty}\,\|u_n\|_{L^{\phi_p} (|x|>R)}=0. \eq Then,
there exist a sequence of scales $({\alpha}_n^{(j)})$, a  sequence
of cores $({x}_n^{(j)})$ and a sequence of profiles $(\psi^{(j)})$
such that the triplets $({\alpha}_n^{(j)}, {x}_n^{(j)},\psi^{(j)})$
are pairwise orthogonal in the sense of Definition \ref{ortho} and,
up to a subsequence extraction, we have for all $\ell\geq 1$,
 \bq \label{decomp}
u_n(x)=\Sum_{j=1}^{\ell}\,\sqrt{\frac{\alpha_n^{(j)}}{2\pi}}\;\psi^{(j)}\left(\frac{-\log|x
- x_n^{(j)}|}{\alpha_n^{(j)}}\right)+{\rm r}_n^{(\ell)}(x),\eq
 with $\Limsup_{n\to\infty}\;\|{\rm
r}_n^{(\ell)}\|_{L^{\phi_p}}\stackrel{\ell\to\infty}\longrightarrow
0.$ Moreover, we have the following stability estimate \bq
\label{orto} \|\nabla
u_n\|_{L^2}^2=\Sum_{j=1}^{\ell}\,\|{\psi^{(j)}}'\|_{L^2}^2+\|\nabla
r_n^{(\ell)}\|_{L^2}^2+\circ(1),\quad n\to\infty. \eq
\end{thm}
\begin{rems}\quad\\
\vspace{-0.5cm}
\begin{itemize}
\item Arguing as in \cite{Bahouri}, we can easily prove that
\begin{equation}\label{gnj} \|g_n\|_{L^{\phi_p}}\stackrel{n\rightarrow\infty}\longrightarrow
 \frac{1}{\sqrt{4\pi}}\,\max_{s>0}\;\frac{|\psi(s)|}{\sqrt{s}},
\end{equation}
where
\begin{equation*}
g_n(x) := \sqrt{\frac{\alpha_n}{2\pi}}\;\psi\left(\frac{-\log|x -
x_n|}{\alpha_n}\right)\cdot
\end{equation*}
Indeed setting $L=\Liminf_{n\rightarrow\infty}\|g_n\|_{L^{\phi_p}}$,
we have for fixed $\varepsilon>0$ and $n$ sufficiently large (up to
subsequence extraction)
$$\Int_{\R^2}\Big({\rm e}^{\big|\frac{g_n(x+x_n)}{L+\varepsilon}\big|^2}-\Sum_{k=0}^{p-1}\Frac{|g_n(x+x_n)|^{2k}}{(L+\varepsilon)^{2k}k!}\Big)\,dx\leq \kappa.$$
Therefore,
$$\Int_{\R^2}\Big({\rm e}^{\big|\frac{g_n(x+x_n)}{L+\varepsilon}\big|^2}-1\Big)\,dx\lesssim \kappa+\Sum_{k=1}^{p-1}\|g_n\|_{L^{2k}}^{2k},$$
which implies in view of \eqref{lim gnj} that
$$\Int_{\R^2}\Big({\rm e}^{\big|\frac{g_n(x+x_n)}{L+\varepsilon}\big|^2}-1\Big)\,dx=2\pi\Int_0^{+\infty}\alpha_n{\rm e}^{2\alpha_ns\Big[\frac{1}{4\pi(L+\varepsilon)^2}\big(\frac{\psi(s)}{\sqrt{s}}\big)^2-1\Big]}\,ds-\pi\lesssim 1.$$
Using the fact that $\psi$ is a continuous function, we deduce that
$$L+\varepsilon\geq\frac{1}{\sqrt{4\pi}}\Max_{s>0}\Frac{|\psi(s)|}{\sqrt{s}},$$
which ensures that
$$L\geq\frac{1}{\sqrt{4\pi}}\Max_{s>0}\Frac{|\psi(s)|}{\sqrt{s}}\cdot$$
To end the proof of \eqref{gnj}, it suffices to establish that for
any $\delta>0$
$$\Int_{\R^2}\Big({\rm e}^{\big|\frac{g_n(x+x_n)}{\lambda}\big|^2}-\Sum_{k=0}^{p-1}\Frac{|g_n(x+x_n)|^{2k}}{(\lambda)^{2k}k!}\Big)\,dx\stackrel{n\rightarrow\infty}\longrightarrow 0,$$
where
$\lambda=\frac{1+\delta}{\sqrt{4\pi}}\Max_{s>0}\frac{|\psi(s)|}{\sqrt{s}}\cdot$
Since
$$\Int_{\R^2}\Big({\rm e}^{\big|\frac{g_n(x+x_n)}{\lambda}\big|^2}-\Sum_{k=0}^{p-1}\Frac{|g_n(x+x_n)|^{2k}}{(\lambda)^{2k}k!}\Big)\,dx
\leq\Int_{\R^2}\Big({\rm
e}^{\big|\frac{g_n(x+x_n)}{\lambda}\big|^2}-1\Big)\,dx,$$ the result
derives immediately from Proposition $1.15$ in \cite{Bahouri}, which
achieves the proof of the result.
\item Applying the same lines of reasoning as in  the proof of Proposition 1.19 in \cite{Bahouri}, we obtain the following result:
\begin{prop}
\label{sum} Let
$\big((\alpha_n^{(j)}),(x_n^{(j)}),\psi^{(j)}\big)_{1\leq j
\leq\ell}$ be a family of triplets of scales, cores and profiles
such that the scales are pairwise orthogonal. Then for any integer
$p$ larger than $1$, we have
$$ \Big\|\Sum_{j=1}^{\ell}\,g_n^{(j)}\Big\|_{L^{\phi_p}}\stackrel{n\rightarrow\infty}\longrightarrow \sup_{1\leq
j\leq\ell}\,\left(\lim_{n\to\infty}\,\big\|g_n^{(j)}\big\|_{L^{{\phi}_p}}\right)\;
,
$$
where the functions $g^{(j)}_n$ are defined by \eqref{def gnj}.
\end{prop}
\end{itemize}
\end{rems}

 As we will see in Section 2, it turns out that the heart of the matter  in the proof of Theorem \ref{theorem} is reduced to the following result concerning the radial case:
 \begin{thm}
\label{theorem-rad} Let $p$ be an integer strictly larger than $1$
and $(u_n)$ be a bounded sequence in $H_{rad}^1(\R^2)$ such that \bq
\label{assum1-rad} u_n\rightharpoonup 0\quad\quad\mbox{and} \eq
 \bq \label{assum2-rad}
\limsup_{n\to\infty}\|u_n\|_{L^{\phi_p}}=A_0 >0. \eq
 Then, there
exist a sequence of pairwise orthogonal scales $({\alpha}_n^{(j)})$
and a sequence of profiles $(\psi^{(j)})$  such that up to a
subsequence extraction, we have for all $\ell\geq 1$,\bq
\label{dec-rad}
u_n(x)=\Sum_{j=1}^{\ell}\,\sqrt{\frac{\alpha_n^{(j)}}{2\pi}}\;\psi^{(j)}\left(\frac{-\log|x|}{\alpha_n^{(j)}}\right)+{\rm
r}_n^{(\ell)}(x),\quad\limsup_{n\to\infty}\;\|{\rm
r}_n^{(\ell)}\|_{L^{\phi_p}}\stackrel{\ell\to\infty}\longrightarrow
0. \eq Moreover, we have the following stability estimate $$
\|\nabla
u_n\|_{L^2}^2=\Sum_{j=1}^{\ell}\,\|{\psi^{(j)}}'\|_{L^2}^2+\|\nabla
r_n^{(\ell)}\|_{L^2}^2+\circ(1),\quad n\to\infty.$$
\end{thm}

\eject

\begin{rems}\quad\\
\vspace{-0.5cm}
\begin{itemize}
\item
Compared with the analogous result concerning the Sobolev embedding
of $H_{rad}^1(\R^2)$into $L^{\phi_1}$ established in \cite{Bahouri},
the hypothesis of compactness at infinity is not required. This is
justified by the fact that $H^1_{rad}(\R^2)$ is compactly embedded
in $L^q(\R^2)$ for any $2<q<\infty$ which implies that \bq\label{un}
\Lim_{n\rightarrow\infty}\|u_n\|_{L^{q}(\R^2)}=0,\quad\forall\,2<q<\infty.
\eq
\item In view of  Proposition \ref{sum}, Theorem \ref{theorem-rad} yields to
$$\|u_n\|_{L^{\phi_p}}\to\sup_{j \geq
1}\,\left(\lim_{n\to\infty}\,\|g_n^{(j)}\|_{L^{\phi_p}}\right),$$
which implies that the first profile in Decomposition
\eqref{dec-rad} can be chosen such that up to extraction \bq
\label{A0}
A_0:=\Limsup_{n\rightarrow\infty}\|u_n\|_{L^{\phi_p}}=\Lim_{n\rightarrow\infty}
\left\|\sqrt{\frac{\alpha_n^{(1)}}{2\pi}}\psi^{(1)}\left(-\frac{\log|x|}{\alpha_n^{(1)}}\right)\right\|_{L^{\phi_p}}.
\eq
\end{itemize}
\end{rems}

Note that the description of the lack of compactness in other
critical Sobolev embeddings was  achieved in \cite{BCG, BZ, Ge2} and
has been at the origin of several prospectus.  Among others, one can
mention \cite{BG,BG2, BIP,BC,km}.
\subsection{Layout of the paper}
Our paper is organized as follows: in Section 2, we establish the algorithmic construction of the decomposition stated in Theorem \ref{theorem}. Then, we study in Section 3 a nonlinear two-dimensional wave equation with the exponential nonlinearity $u\,\phi_p(\sqrt{4\pi}u)$. Firstly, we prove the global well-posedness and the scattering in the energy space both in the subcritical and critical cases, and secondly we compare the evolution of this equation with the evolution of the solutions of the free Klein-Gordon equation in the same space.\\
\\
We mention that $C$ will be used to denote a constant which may vary
from line to line. We also use $A \lesssim B$ to denote an estimate
of the form $A \leq CB$ for some absolute constant $C$ and $A
\approx B$ if $A \lesssim B$ and $B \lesssim A$. For simplicity, we
shall also still denote by $(u_n)$ any subsequence of $(u_n)$ and
designate by $\circ(1)$ any sequence which tends to $0$ as $n$ goes
to infinity.

\section{Proof of Theorem \ref{theorem}}
\subsection{Strategy of the proof}
The proof of Theorem \ref{theorem} uses in a crucial way capacity
arguments and is done in three steps: in the first step,  we begin
by the study of $u^\ast_n$ the symmetric decreasing rearrangement of
$u_n$. This led us to establish  Theorem \ref{theorem-rad}. In the
second step, by a technical process developed in \cite{Bahouri2}, we
reduce ourselves  to one scale and extract  the first core
$(x_n^{(1)})$ and the first profile $\psi^{(1)}$ which enables us to
extract the first  element $
\sqrt{\frac{\alpha_n^{(1)}}{2\pi}}\;\psi^{(1)}\left(\frac{-\log|x -
x_n^{(1)}|}{\alpha_n^{(j)}}\right)$. The third step is devoted to
the study of the remainder term.  If the limit of its Orlicz norm is
null we stop the process. If not, we prove that this remainder term
satisfies the same properties as the sequence we  start with  which
allows us to  extract a second  elementary concentration
concentrated around a second core $(x_n^{(2)})$. Thereafter, we
establish the property of orthogonality between the  first  two
elementary concentrations and finally we prove that this process
converges.

 \subsection{Proof of Theorem \ref{theorem-rad}}
The main ingredient in the proof of Theorem \ref{theorem-rad}
consists to extract a scale and a profile $\psi$ such that
\begin{equation}\label{psi'}
\|\psi'\|_{L^2(\R)}\geq C A_0,
\end{equation}
where $C$ is a universal constant. To go to this end, let us for a
bounded sequence $(u_n)$ in $H^1_{rad}(\R^2)$ satisfying the
assumptions \eqref{assum1-rad} and \eqref{assum2-rad}, set $v_n(s) =
u_n({\rm e}^{-s})$. Combining \eqref{un} with the following
well-known radial estimate:
$$|u(r)|\leq \frac{C}{r^\frac{1}{p+1}}\|u\|_{L^{2p}}^\frac{p}{p+1}\|\nabla u\|_{L^2}^\frac{1}{p+1}$$
where $r=|x|$, we infer that
\begin{equation}\label{vn}
\Lim_{n\rightarrow\infty}\|v_n\|_{L^\infty(]-\infty,M])}=0,\quad\forall
M\in\R.
\end{equation}
This gives rise to  the following result:
\begin{prop}
\label{step1} For any $\de>0$, we have
\begin{equation}
\label{depart} \sup_{s\geq
0}\left(\Big|\frac{v_n(s)}{A_0-\de}\Big|^2-s\right)\to\infty,\quad
n\to\infty.
\end{equation}
\end{prop}
\begin{proof}
We proceed by contradiction. If not, there exists $\de>0$ such that,
up to a subsequence  extraction
\begin{equation}
\label{dominate} \sup_{s\geq 0,
n\in\N}\;\;\left(\Big|\frac{v_n(s)}{A_0-\de}\Big|^2-s\right)\leq
C<\infty.
\end{equation}
On the one hand, thanks to \eqref{vn} and \eqref{dominate}, we get
by virtue of Lebesgue theorem
\begin{eqnarray*}
\int_{|x|<1}\;\left({\rm
e}^{|\frac{u_n(x)}{A_0-\de}|^2}-\Sum_{k=0}^{p-1}\frac{|u_n(x)|^{2k}}{(A_0-\delta)^{2k}k!}\right)\,dx&\leq&
\int_{|x|<1}\;\left({\rm
e}^{|\frac{u_n(x)}{A_0-\de}|^2}-1\right)\,dx\\
&\leq&2\pi\,\int_{0}^\infty\;\left({\rm
e}^{|\frac{v_n(s)}{A_0-\de}|^2}-1\right)\,{\rm
e}^{-2s}\,ds\stackrel{n\to\infty}\longrightarrow 0.
\end{eqnarray*}
On the other hand, using Property \eqref{vn} and the simple fact
that for any positive real number $M$, there exists a finite
constant $C_{M,p}$ such that
$$
\sup_{|t|\leq M}\,\left(\frac{{\rm
e}^{t^2}-\sum_{k=0}^{p-1}\frac{t^{2k}}{k!}}{t^{2p}}\right)<C_{M,p},
$$
we deduce in view of  \eqref{un}  that
$$
\int_{|x|\geq 1}\;\left({\rm
e}^{|\frac{u_n(x)}{A_0-\delta}|^2}-\Sum_{k=0}^{p-1}\frac{|u_n(x)|^{2k}}{(A_0-\delta)^{2k}k!}\right)\,dx\lesssim
\|u_n\|_{L^{2p}}^{2p}\to 0\,.
$$
 Consequently,
$$
\Limsup_{n\to\infty}\,\|u_n\|_{L^{\phi_p}}\leq A_0-\delta,
$$
which is in contradiction with Hypothesis \eqref{assum2-rad}.
\end{proof}
An immediate consequence of the previous proposition is the
following corollary whose proof is identical to the proof of
Corollaries 2.4 and  2.5 in \cite{Bahouri}.
\begin{cor}\label{alpha}
 Under the above notations, there exists a sequence $(\alpha_n^{(1)})$ in $\R_+$
tending to infinity such that
\begin{equation}\label{alphan1}
4\,\Big|\frac{v_n(\alpha_n^{(1)})}{A_0}\Big|^2-\alpha_n^{(1)}\stackrel{n\to\infty}\longrightarrow\infty
\end{equation}
and for $n$ sufficiently large, there exists a positive constant $C$
such that \bq \label{bo} \frac{ A_0}{2}\sqrt{\alpha_n^{(1)}}\leq|
v_n(\alpha_n^{(1)})|\leq C \sqrt{\alpha_n^{(1)}}+\circ(1). \eq
\end{cor}
Now, setting
$$\psi_n(y)=\sqrt{\frac{2\pi}{\alpha_n^{(1)}}} v_n(\alpha_n^{(1)} y),$$
we obtain along the same lines as in Lemma 2.6 in \cite{Bahouri} the
following result:
\begin{lem} Under notations of Corollary \ref{alpha}, there exists a profile $\psi^{(1)}\in \mathcal{P}$ such that, up to a subsequence extraction
$$\psi_n'\rightharpoonup (\psi^{(1)})'\; in\; L^2(\R) \quad and \quad \|(\psi^{(1)})'\|_{L^2}\geq\sqrt{\frac{\pi}{2}} A_0.$$
\end{lem}
To achieve the proof of Theorem \ref{theorem-rad}, let us consider
the remainder term
\begin{equation}\label{def}
r_n^{(1)}(x)=u_n(x)-g_n^{(1)}(x),
\end{equation}
where
$$g_n^{(1)}(x)=\sqrt{\frac{\alpha_n^{(1)}}{2\pi}}\psi^{(1)}\left(\frac{-\log|x|}
{\alpha_n^{(1)}}\right).$$ By straightforward computations, we can
easily prove that $(r^{(1)}_n)$ is bounded in $H^1_{rad}(\R^2)$ and
satisfies the hypothesis \eqref{assum1-rad} together with the
following property:
\begin{equation}\label{r1}
\Lim_{n\rightarrow \infty}\|\nabla
r_n^{(1)}\|_{L^2(\R^2)}^2=\underset{n\rightarrow
\infty}{\lim}\|\nabla
u_n\|_{L^2(\R^2)}^2-\big\|(\psi^{(1)})'\big\|_{L^2(\R)}^2.
\end{equation}
Let us now define
$A_1=\underset{n\rightarrow\infty}{\limsup}\|r_n^{(1)}\|_{L^{\phi_p}}$.
If $A_1=0$, we stop the process. If not, arguing as above, we prove
that there exist a scale $(\alpha_n^{(2)})$ satisfying the statement
of Corollary \ref{alpha} with $A_1$ instead of $A_0$ and a profile
$\psi^{(2)}$ in $\mathcal{P}$ such that
 $$r_n^{(1)}(x)=\sqrt{\frac{\alpha_n^{(2)}}{2\pi}} \psi^{(2)}\left(\frac{-\log|x|}{\alpha_n^{(2)}}\right)+r_n^{(2)}(x),$$
 with $\|(\psi^{(2)})'\|_{L^2}\geq\frac{\sqrt{2\pi}}{2}A_1$ and
\begin{equation*}
\Lim_{n\rightarrow \infty}\|\nabla
r_n^{(2)}\|_{L^2(\R^2)}^2=\underset{n\rightarrow
\infty}{\lim}\|\nabla
r_n^{(1)}\|_{L^2(\R^2)}^2-\big\|(\psi^{(2)})'\big\|_{L^2(\R)}^2.
\end{equation*}
Moreover, as in \cite{Bahouri} we can show that $(\alpha_n^{(1)})$ and $(\alpha_n^{(2)})$ are orthogonal. Finally, iterating the process, we get at step $\ell$
$$
u_n(x)=\Sum_{j=1}^{\ell}\,\sqrt{\frac{\alpha_n^{(j)}}{2\pi}}\;\psi^{(j)}\left(\frac{-\log|x|}{\alpha_n^{(j)}}\right)+{\rm
r}_n^{(\ell)}(x),
$$
with
$$
\limsup_{n\to\infty}\,\|r^{(\ell)}_n\|_{H^1}^2\lesssim
1-A_0^2-A_1^2-\cdots -A_{\ell-1}^2\,,
$$
which implies that $A_\ell\to 0$ as $\ell\to\infty$ and ends the
proof of the theorem.

\subsection{Extraction of the cores and profiles} This step is performed as the proof of Theorem 1.16 in \cite{Bahouri1}.
We sketch it here briefly for the convenience of the reader. Let $u_n^\ast$ be the symmetric decreasing rearrangement of $u_n$. Since $u^\ast_n\in H^1_{rad}(\R^2)$ and satisfies the assumptions of  Theorem \ref{theorem-rad}, we infer that there exist a sequence $(\alpha_n^{(j)})$  of pairwise orthogonal scales and a sequence of profiles $(\varphi^{(j)})$ such that, up to subsequence extraction,
$$
u_n^*(x)=\Sum_{j=1}^{\ell}\,\sqrt{\frac{\alpha_n^{(j)}}{2\pi}}\;\varphi^{(j)}\left(\frac{-\log|x|}{\alpha_n^{(j)}}\right)+{\rm
r}_n^{(\ell)}(x),\quad\limsup_{n\to\infty}\;\|{\rm
r}_n^{(\ell)}\|_{L^{\phi_p}}\stackrel{\ell\to\infty}\longrightarrow
0. $$ Besides, in view of   \eqref{A0}, we can assume that
$$A_0=\underset{n\rightarrow \infty}{\lim} \left\| \sqrt{\frac{\alpha_n^{(1)}}{2\pi}} \varphi^{(1)}\left(-\frac{\log|x|}{\alpha_n^{(1)}}\right)\right\|_{L^{{\Phi}_p}}.$$
Now to extract the cores and profiles, we shall firstly reduce to
the case of one scale according to Section 2.3 in \cite{Bahouri2},
where a   suitable truncation of $u_n$ was introduced.   Then
assuming that $$ u_n^*(x) =
\sqrt{\frac{\alpha_n^{(1)}}{2\pi}}\;\varphi^{(1)}\left(\frac{-\log|x|}{\alpha_n^{(1)}}\right),$$
we apply the strategy developed in Section 2.4 in \cite{Bahouri2} to
extract the cores and the profiles. This approach is based on
capacity arguments: to carry out the extraction process of mass
concentrations, we prove by contradiction that if the mass
responsible for the lack of compactness of the Sobolev embedding in
the Orlicz space is scattered, then the energy used would exceed
that of the starting sequence. This main point can be formulated in
the following terms:
\begin{lem} [ Lemma 2.5 in \cite{Bahouri2}]   \label{concentration}
There exist $\delta_0>0$ and $N_1\in \N$ such that for any $n\geq
N_1$ there exists $x_n$ such that
\begin{equation}\label{concentration0}
   \frac{|E_n\cap B(x_n,\rm{e}^{-b \alpha_n^{(1)}})|}{|E_n|}\geq\delta_0 A_0^2,
\end{equation}
where $E_n:=\{x\in \R^2; |u_n(x)|\geq
\sqrt{2\alpha_n^{(1)}}(1-\frac{\varepsilon_0}{10})A_0\}$ with
$0<\varepsilon_0<\frac{1}{2}$, $B(x_n,\rm{e}^{-b\,\alpha_n^{(1)}})$
designates the ball of center $x_n$ and radius
$\rm{e}^{-b\,\alpha_n^{(1)}}$ with $b=1-2\varepsilon_0$ and $|.|$
denotes the Lebesgue measure.
\end{lem}
Once extracting the first core $(x_n^{(1)})$ making use of the
previous lemma, we  focus on the extraction of the first profile.
For that purpose, we consider the sequence
$$\psi_n(y,\theta)=\sqrt{\frac{2\pi}{\alpha_n^{(1)}}} v_n(\alpha_n^{(1)}y,\theta),$$
where
$v_n(s,\theta)=(\tau_{x_n^{(1)}}u_n)(\rm{e}^{-s}\cos\theta,\rm{e}^{-s}\sin
\theta)$ and $(x_n^{(1)})$ satisfies
$$\frac{|E_n\cap B(x_n,\rm{e}^{-(1-2\varepsilon_0)\alpha_n^{(1)}}|}{|E_n|}\geq\delta_0 A_0^2.$$
Taking advantage of the invariance of Lebesgue measure  under
translations, we deduce that \begin{eqnarray*}\label{property}
  \|\nabla u_n\|_{L^2}^2 &=& \frac{1}{2\pi}\int_{\R}\int_{0}^{2\pi} |\partial_y \psi_n(y,\theta)|^2 dy d\theta \\
   &+& \frac{\alpha_n^{(1)}}{2\pi}\int_{\R}\int_{0}^{2\pi} |\partial_{\theta}\psi_n(y,\theta)|^2 dy d\theta.
\end{eqnarray*}
  Since the scale $\alpha_n^{(1)}$ tends to infinity and the sequence $(u_n)$ is bounded in $H^1(\R^2)$, this implies that up to a subsequence extraction
 $\partial_{\theta}\psi_n\underset{n\rightarrow \infty}{\rightarrow}0$ and $\partial_{y}\psi_n \underset{n\rightarrow \infty}{\rightharpoonup}g$ in $L^2(\R\times[0,2\pi])$, where $g$ only depends on the variable $y$. Thus  introducing the function
  $$\psi^{(1)}(y)=\int_{0}^{y} g(\tau) d\tau,$$
  we obtain along the same lines as in Proposition 2.8 in \cite{Bahouri2} the following result:
  \begin{prop}
  The function $\psi^{(1)}$ belongs to the set of profiles $\mathcal{P}$. Besides   for any $y \in \R$, we have
\begin{equation}
\label{nonradcvprofile} \frac{1}{2\pi} \int^{2\pi}_{0} \psi_n
(y,\theta)\, d\theta \to \psi^{(1)} (y), \end{equation} as~$n$ tends
to infinity and there exists an absolute constant $C$  such that
  \begin{equation}\label{estimate}
    \|{\psi^{(1)}}'\|_{L^2}\geq C \, A_0.
  \end{equation}
  \end{prop}
  \subsection{End of the proof}  To achieve the proof of the theorem, we argue exactly as in Section 2.5 in \cite{Bahouri2} by iterating the process exposed in the previous section. For that purpose, we set
   $$r_n^{(1)}(x)=u_n(x)-g_n^{(1)}(x),$$
  where
$$g_n^{(1)}(x)=\sqrt{\frac{\alpha_n^{(1)}}{2\pi}}\psi^{(1)}\left(-\frac{\log|x-x_n^{(1)}|}{\alpha^{(1)}_n}\right).$$
One can easily check that the sequence $(r_n^{(1)})$ weakly
converges to $0$ in $H^1(\R^2)$. Moreover, since
$\psi^{(1)}_{|]-\infty,0]}=0$, we have for any $R\geq 1$
  \begin{equation}\label{us}\|r_n^{(1)}\|_{L^{\Phi_p}(|x-x_n^{(1)}|\geq R)}=\|u_n\|_{L^{\Phi_p}(|x-x_n^{(1)}|\geq R)}.  \end{equation}
But by assumption, the sequence $(u_n)$  is compact at infinity in
the Orlicz space $L^{\Phi_p}$.  Thus the core $(x_n^{(1)})$ is
bounded in $\R^2$, which ensures  in view of \eqref{us} that
$(r_n^{(1)})$ satisfies the hypothesis of compactness at infinity
\eqref{assum3}.  Finally, taking advantage of the weak convergence
of $(\partial_y \psi_n)$ to ${\psi^{(1)}}'$ in $L^2(y,\theta)$ as
$n$ goes to infinity,  we get
  $$\underset{n\rightarrow\infty}{\lim}\|\nabla r_n^{(1)}\|_{L^2}^2=\underset{n\rightarrow\infty}{\lim}\|\nabla u_n^{(1)}\|_{L^2}^2-\|{\psi^{(1)}}'\|_{L^2}^2.$$
 Now, let us define $A_1:=\underset{n\rightarrow\infty}{\limsup}\|r_n^{(1)}\|_{L^{\Phi_p}}$. If $A_1=0$, we stop the process. If not, knowing that $(r_n^{(1)})$ verifies the assumptions of Theorem \ref{theorem}, we apply the above reasoning, which gives rise to  the existence of a scale $(\alpha_n^{(2)})$, a core $(x^{(2)}_n)$  satisfying the statement of Lemma \ref{concentration} with $A_1$
instead of $A_0$ and a profile $\psi^{(2)}$ in $\mathcal{P}$ such
that
$$r_n^{(1)}(x)=\sqrt{\frac{\alpha_n^{(2)}}{2\pi}}\psi^{(2)}\left(-\frac{\log|x-x_n^{(2)}|}{\alpha^{(2)}_n}\right)+r_n^{(2)}(x),$$
with $\|{\psi^{(2)}}'\|_{L^2}\geq C\,A_1$ and
$$\underset{n\rightarrow\infty}{\lim}\|\nabla r_n^{(2)}\|_{L^2}^2=\underset{n\rightarrow\infty}{\lim}\|\nabla r_n^{(1)}\|_{L^2}^2-\|{\psi^{(2)}}'\|_{L^2}^2.$$
Arguing as in \cite{Bahouri2}, we show that the triplets
$\big(\alpha_n^{(1)},x_n^{(1)},\psi^{(1)}\big)$ and
$\big(\alpha_n^{(2)},x_n^{(2)},\psi^{(2)}\big)$ are orthogonal in
the sense of Definition \ref{ortho} and prove that the process of
extraction of the elementary concentration converges. This ends the
proof of Decomposition \eqref{noraddecomp}. The orthogonality
equality \eqref{ortogonal} derives immediately from Proposition 2.10
in \cite{Bahouri2}. The proof of Theorem \ref{theorem} is then
achieved.

\section{Nonlinear wave equation}
\subsection{Statement of the results}
In this section, we investigate the initial value problem for the
following nonlinear wave equation:
\begin{eqnarray}\label{exp_p}
\left\{\begin{array}{lll}
 \square u +u+u\,\left({\rm e}^{4\pi u^2}-\Sum_{k=0}^{p-1}\frac{(4\pi)^k u^{2k}}{k!}\right)=0,\\\\
 u(0)=u_0\in H^1(\R^2),\quad\partial_tu(0)=u_1\in L^2(\R^2),
\end{array}
\right.
\end{eqnarray}
where $p\geq 1$ is an integer, $u=u(t,x)$ is a real-valued function of $(t,x)\in\R\times \R^2$ and $\square=\partial^2_t-\Delta$ is the wave operator.\\

Let us recall that in \cite{Ibrahim,Ibrahim1}, the authors proved
the global well-posedness for the Cauchy problem \eqref{exp_p} when
$p=1$ and the scattering
when $p=2$ in the subcritical and critical cases (i.e when the energy is less or equal to some threshold). Note also that in \cite{Struwe1, Struwe2}, M. Struwe constructed global smooth solutions to \eqref{exp_p} with smooth data of arbitrary size in the case $p=1$.\\

Formally, the solutions of the Cauchy problem \eqref{exp_p} satisfy
the following conservation law:
\begin{eqnarray}
\label{energy}\quad\quad E_p(u,t)&:=&\|\partial_t
u(t)\|_{L^2}^2+\|\nabla u(t)\|_{L^2}^2+\frac{1}{4\pi} \left\|{\rm
e}^{4\pi
u(t)^2}-1-\Sum_{k=2}^p\frac{(4\pi)^k}{k!}u(t)^{2k}\right\|_{L^1}\\\nonumber
&=&E_p(u,0):=E_p^0.
\end{eqnarray}
This conducts us, as in \cite{Ibrahim}, to define the notion of
criticality in terms of the size of the initial energy $E_p^0$ with
respect to $1$.
\begin{defi}
\label{d1} The Cauchy problem \eqref{exp_p} is said to be
subcritical if
$$
E_p^0<1.
$$
It is said to be critical if $E_p^0=1$ and supercritical if
$E_p^0>1$.
\end{defi}
We shall prove the following result:
\begin{thm}\label{solution}
Assume that $E_p^0\leq 1.$ Then the Cauchy problem \eqref{exp_p} has
a unique global solution $u$ in the space
$$\mathcal{ C}(\R, H^1(\R^2))\cap\mathcal{ C}^1(\R,L^2(\R^2)).$$
 Moreover, $u\in L^4(\R,\mathcal{ C}^{1/4})$ and scatters.
\end{thm}
\subsection{Technical tools}
The proof of Theorem \ref{solution} is based on priori estimates.
This requires the control of the nonlinear term
\begin{equation}\label{nonlinear}
F_p(u):=u\,\left({\rm e}^{4\pi u^2}-\Sum_{k=0}^{p-1}\frac{(4\pi)^k
u^{2k}}{k!}\right)
\end{equation}
in $L^1_t(L^2_x)$. To achieve our goal, we will resort to Strichartz
estimates for the 2D Klein-Gordon equation. These estimates, proved
in \cite{Ginibre}, state as follows:
\begin{prop}
\label{admiss} Let $T>0$ and $(q,r)\in[4,\infty]\times[2,\infty]$ an
admissible pair, i.e $$\frac{1}{q}+\frac{2}{r}=1.$$ Then,
\begin{equation}
\label{e6} \|v\|_{L^q([0,T], {\mathrm
B}^{1}_{r,2}(\R^2))}\lesssim\Big[\|v(0)\|_{H^1(\R^2)}+\|\partial_tv(0)\|_{L^2(\R^2)}+\|\square
v+v\|_{L^1([0,T],L^2(\R^2))}\Big],
\end{equation}
where ${\mathrm B}^{1}_{r,2}(\R^2)$ stands for the usual
inhomogeneous Besov space (see for example \cite{chemin} or
\cite{RS} for a detailed exposition  on Besov spaces).
\end{prop}
Noticing that $(q,r) = (4, 8/3)$ is an admissible pair and recalling
that
$${\mathrm B}^{1}_{8/3,2}(\R^2)\hookrightarrow \mathcal{ C}^{1/4}(\R^2),$$
 we deduce that
\begin{equation}
\label{str}\|v\|_{L^4([0,T],\mathcal{C}^{1/4}(\R^2))}\lesssim\Big[\|v(0)\|_{H^1(\R^2)}+\|\partial_tv(0)\|_{L^2(\R^2)}+\|\square
v+v\|_{L^1([0,T],L^2(\R^2))}\Big].
\end{equation}
To control the nonlinear term $F_p(u)$ in $L^1_t(L^2_x)$, we will
make use of the following logarithmic inequalities proved in
\cite[Theorem 1.3]{DlogSob}.
\begin{prop}
\label{Hmu}
  For any $\lambda>\frac{2}{\pi}$ and
any $0<\mu\leq 1$, a constant $C_{\lambda}>0$ exists such that for
any function $u$ in $H^1(\R^2)\cap \mathcal{C}^{1/4}(\R^2)$, we have
\begin{equation}
\label{H-mu} \|u\|^2_{L^\infty}\leq
\lambda\|u\|_{\mu}^2\log\left(C_{\lambda,\mu} +
\frac{2\|u\|_{{\mathcal C}^{1/4}}}{\|u\|_{\mu}}\,\right),
\end{equation}
where $\|u\|_{\mu}^2:=\|\nabla u\|_{L^2}^2+\mu^2\|u\|_{L^2}^2$.
\end{prop}
\subsection{Proof of Theorem \ref{solution}}
The proof of this result, divided into three steps, is inspired from
the proofs of Theorems $1.8$, $1.11$, $1.12$ in \cite{Ibrahim} and
Theorem $1.3$ in \cite{Ibrahim1}.
\subsubsection{Local existence}
 Let us start by proving the local existence to the Cauchy problem \eqref{exp_p} in the case where $\|\nabla u_0\|_{L^2(\R^2)}< 1$. To do so, we use a standard fixed-point argument and introduce for any nonnegative time $T$ the following space:
$$\mathcal{E}_T=\mathcal{C}([0,T],H^1(\R^2))\cap\mathcal{C}^1([0,T],L^2(\R^2))\cap L^4([0,T],\mathcal{C}^{1/4}(\R^2))$$
endowed with the norm
$$\|u\|_T:=\underset{0\leq t\leq T}{\sup}\Big[\|u(t)\|_{H^1}+\|\partial_t u(t)\|_{L^2}\Big]+\|u\|_{L^4([0,T],\mathcal{C}^{1/4})}.$$
 For a positive time $T$ and a positive real number $\delta$, we denote by $\mathcal{E}_T(\delta)$ the ball in the space $\mathcal{E}_T$ of radius $\delta$ and centered at the origin. On this ball, we define the map $\Phi$ by
$$v\longmapsto \Phi(v)=\widetilde{v},$$
where
$$\square \widetilde{v}+\widetilde{v}=-F_p(v+v_0),\quad \widetilde{v}(0)=\partial_t\widetilde{v}(0)=0$$
and $v_0$ is the solution of the free Klein-Gordon equation
$$\square v_0+v_0=0,\quad v_0(0)=u_0,\quad and \quad \partial_tv_0(0)=u_1.$$
Now, the goal is to show that if $\delta$ and $T$ are small enough,
then the map $\Phi$ is well-defined from $\mathcal{E}_T(\delta)$
into itself and it is a contraction. To prove that $\Phi$ is
well-defined, it suffices in view of the Strichartz estimates
\eqref{e6} to estimate $F_p(v+v_0)$ in the space
$L^1([0,T],L^2(\R^2))$. Arguing as in \cite{Ibrahim} and using the
H\"older inequality and the Sobolev embedding, we obtain for any
$\epsilon>0$
\begin{eqnarray*}
\Int_{\R^2}|F_p(v+v_0)|^2\;dx&\leq& \Int_{\R^2}|F_1(v+v_0)|^2\;dx\\
&\lesssim&\|v+v_0\|_{H^1}^2\,{\rm
e}^{4\pi\|v+v_0\|_{L^\infty}^2}\left\|{\rm e}^{4\pi
(v+v_0)^2}-1\right\|_{L^{1+\epsilon}}.
\end{eqnarray*}
 Since $\|\nabla u_0\|_{L^2}<1$, we can choose $\mu>0$ such that $\|u_0\|_\mu<1$. Since $v_0$ is continuous in time, there exist a time $T_0$ and a constant $0<c<1$ such that for any $t$ in $[0,T_0]$ we have
$$\| v_0(t)\|_\mu\leq c.$$
According to Proposition \ref{Hmu}, we infer that
$${\rm e}^{4\pi\|v+v_0\|_{L^\infty}^2}\lesssim\left(1+\frac{\|v+v_0\|_{\mathcal{C}^{1/4}}}{\delta+c}\right)^{8\eta},$$
for some $0<\eta<1$. Besides, applying the Trudinger-Moser
inequality \eqref{Trudinger-Moser} for $p=1$, the fact that
$$4\pi(1+\epsilon)(\delta+c)^2 \longrightarrow 4\pi c<4\pi \quad
\mbox{as}\;\epsilon,\,\delta\rightarrow 0\quad
\mbox{and}\quad\left\|\nabla\left(\frac{v+v_0}{\delta+c}\right)\right\|_{L^2}\leq
1$$ ensures that
\begin{eqnarray*}
\left\|{\rm e}^{4\pi
(v+v_0)^2}-1\right\|_{L^{1+\epsilon}}^{1+\epsilon}
&\leq& C_{\epsilon}\left\|{\rm e}^{4\pi (1+\epsilon)(v+v_0)^2}-1\right\|_{L^1}\\
&\leq&C_{\epsilon,\delta}\|v+v_0\|_{L^2}^2\\
&\leq&C_{\epsilon,\delta}(1+\|u_0\|_{H^1}+\|u_1\|_{L^2})^2.
\end{eqnarray*}
Therefore, for any $0<T\leq T_0$, we obtain that
$$\|F_p(v+v_0)\|_{L^1([0,T],L^2(\R^2))}\lesssim T^{1-\eta}(1+\|u_0\|_{H^1}+\|u_1\|_{L^2})^{4\eta}.$$
Now, to prove that $\Phi$ is a contraction (at least for $T$ small),
let us consider two elements $v_1$ and $v_2$ in
$\mathcal{E}_T(\delta)$. Notice that, for any $\epsilon>0$,
\begin{eqnarray*}
|F_p(v_1+v_0)-F_p(v_2+v_0)| &=&|v_1-v_2|(1+8\pi\overline{v}^2)\left({\rm e}^{4\pi\overline{v}^2}-\Sum_{k=0}^{p-2}\frac{(4\pi)^k\overline{v}^{2k}}{k!}\right)\\
&\leq& C_\epsilon|v_1-v_2|\left({\rm
e}^{4\pi(1+\epsilon)\overline{v}^2}-1\right),
\end{eqnarray*}
where $\overline{v}=(1-\theta)(v_0+v_1)+\theta(v_0+v_2),$ for some
$\theta=\theta(t,x)\in[0,1].$ Using a convexity argument, we get
\begin{eqnarray*}
|F_p(v_1+v_0)-F_p(v_2+v_0)|&\leq& C_\epsilon\left|(v_1-v_2)\left({\rm e}^{4\pi(1+\epsilon)(v_1+v_0)^2}-1\right)\right|\\
&+&C_\epsilon\left|(v_1-v_2)\left({\rm
e}^{4\pi(1+\epsilon)(v_2+v_0)^2}-1\right)\right|.
\end{eqnarray*}
This implies, in view of Strichartz estimates \eqref{str}, that
\begin{eqnarray*}
\|\Phi(v_1)-\Phi(v_2)\|_{T}&\lesssim& \|F_p(v_1+v_0)-F_p(v_2+v_0)\|_{L^1([0,T],L^2(\R^2))}\\
&\leq& C_\epsilon\Int_0^T\left\|(v_1-v_2)\left({\rm e}^{4\pi(1+\epsilon)(v_1+v_0)^2}-1\right)\right\|_{L^2}\,dt\\
&+&  C_\epsilon\Int_0^T\left\|(v_1-v_2)\left({\rm
e}^{4\pi(1+\epsilon)(v_2+v_0)^2}-1\right)\right\|_{L^2}\,dt,
\end{eqnarray*}
which leads along the same lines as above to
\begin{eqnarray*}
  \|\Phi(v_1)-\Phi(v_2)\|_T &\lesssim&  T^{1-(1+\epsilon)\eta}(1+\|u_0\|_{H^1}+\|u_1\|_{L^2})^{4(1+\epsilon)\eta}\|v_1-v_2\|_T.
\end{eqnarray*}
If the parameter $\epsilon$ is small enough, then $(1+\epsilon)\eta<1$ and therefore, for $T$ small enough, $\Phi$ is a contraction map. This implies the uniqueness of the solution in $v_0+\mathcal{E}_T(\delta)$.\\
Now, we shall prove the uniqueness in the energy space. The idea
here is to establish that, if $u=v_0+v$ is a solution of
\eqref{exp_p} in $\mathcal{C}([0,T],H^1(\R^2))\cap
\mathcal{C}^1([0,T],L^2(\R^2))$, then necessarily $v\in
\mathcal{E}_T(\delta)$ at least for $T$ small. Starting from the
fact that $v$ satisfies
$$\square v +v=-F_p(v+v_0),\quad v(0)=\partial_tv(0)=0,$$
we are reduced, thanks to the Strichartz estimates \eqref{e6}, to
control the term $ F_p(v+v_0)$ in the space $L^1([0,T],L^2(\R^2)).$
But $|F_p(v+v_0)|\leq |F_1(v+v_0)|$, which leads to the result
arguing exactly as in \cite{Ibrahim}.

\subsubsection{Global existence}
In this section, we shall establish that our solution is global in
time both in subcritical and critical cases. Firstly, let us notice
that the assumption $E_p^0\leq 1$ implies that $\|\nabla u_0
\|_{L^2(\R^2)}<1$, which ensures in view of Section 3.3.1 the
existence of a unique maximal solution $u$ defined on $[0,T^*)$
where $0<T^*\leq\infty$ is the lifespan of $u$. We shall
proceed by contradiction assuming that $T^*<\infty$. In the
subcritical case, the conservation law \eqref{energy} implies that
 $$ \Sup_{t\in(0,T^*)}\|\nabla u(t)\|_{L^2(\R^2)}<1.$$
Let then $0<s<T^*$ and consider the following Cauchy problem:
\begin{equation}\label{v}
\square v+v+F_p(v)=0,\quad v(s)=u(s),\quad\mbox{and}\quad
\partial_tv(s)=\partial_tu(s).
\end{equation}
As in the first step of the proof, a fixed-point argument ensures the existence of $\tau>0$ and a unique solution $v$ to \eqref{v} on the interval $[s,s+\tau]$. Noticing that $\tau$ does not depend on $s$, we can choose $s$ close to $T^*$ such that $T^*-s<\tau$. So, we can prolong the solution $u$ after the time $T^*$, which is a contradiction.\\
In the critical case, we cannot apply the previous argument because
it is possible that the following concentration phenomenon holds:
\begin{equation}\label{concent}
\Limsup_{t\rightarrow T^*}\|\nabla u(t)\|_{L^2(\R^2)}=1.
\end{equation}
In fact, we shall show that \eqref{concent} cannot hold in this
case. To go to this end, we argue as in the proof of Theorem $1.12$
in \cite{Ibrahim}. Firstly, since the first equation of the Cauchy
problem \eqref{exp_p} is invariant under time translation, we can
assume that $T^*=0$ and that the initial time is $t=-1$. Similarly
to \cite[Proposition 4.2, Corollary 4.4]{Ibrahim}, it follows that
the maximal solution $u$ satisfies
\begin{equation}\label{a1}
\Limsup_{t\rightarrow 0^-}\|\nabla u(t)\|_{L^2(\R^2)}=1,
\end{equation}
\begin{equation}\label{a2}
\Lim_{t\rightarrow 0^-}\|u(t)\|_{L^2(\R^2)}=0,
\end{equation}
\begin{equation}\label{a3}
\Lim_{t\rightarrow 0^-}\Int_{|x-x^*|\leq -t}|\nabla
u(t,x)|^2\;dx=1,\quad\mbox{and}
\end{equation}
\begin{equation}\label{a4}
\forall t<0,\quad\Int_{|x-x^*|\leq -t}e_p(u)(t,x)\;dx=1,
\end{equation}
for some $x^*\in \R^2$, where $e_p(u)$ denotes the energy density
defined by
$$e_p(u)(t,x):=(\partial_t u)^2+|\nabla u|^2+\frac{1}{4\pi}\left({\rm e}^{4\pi u^2}-1-\Sum_{k=2}^{p}\Frac{(4\pi)^ku^{2k}}{k!}\right).$$
 Without loss of generality, we can assume that $x^*=0$, then multiplying the equation of the problem \eqref{exp_p} respectively by $\partial_tu$ and $u$, we obtain formally
\begin{equation}\label{b1}
\partial_te_p(u)-div_x(2\partial_tu\nabla u)=0,
\end{equation}
\begin{equation}\label{b2}
\partial_t(u\partial_tu)-div_x(u\nabla u)+|\nabla u|^2 -|\partial_t u|^2+u^2{\rm e}^{4\pi u^2}-\Sum_{k=1}^{p-1}\Frac{(4\pi)^ku^{2k+2}}{k!}=0.
\end{equation}
Integrating the conservation laws \eqref{b1} and \eqref{b2} over the
backward truncated cone
$$K_S^T:=\Big\{(t,x)\in\R\times \R^2 \;\mbox{ such that}\;S\leq t\leq T\;\mbox{and}\;|x|\leq -t\Big\}$$
for $S<T<0$, we get
\begin{equation}\label{c1}
\Int_{B(-T)}e_p(u)(T,x)\;dx-\Int_{B(-S)}e_p(u)(S,x)\;dx
\end{equation}
\begin{equation*}
 =\frac{-1}{\sqrt{2}}\Int_{M_S^T}\left[\left|\partial_t u\frac{x}{|x|}+\nabla u\right|^2+\frac{1}{4\pi}\left({\rm e}^{4\pi u^2}-1-\Sum_{k=2}^p\frac{(4\pi)^ku^{2k}}{k!}\right)\;dx\,dt\right],
\end{equation*}
\begin{equation}\label{c2}
\Int_{B(-T)}\partial_tu(T)u(T)\;dx-\Int_{B(-S)}\partial_t
u(S)u(S)\;dx+\frac{1}{\sqrt{2}}\Int_{M_S^T}\left(\partial_t u+\nabla
u.\frac{x}{|x|}\right)u\;dx\,dt
\end{equation}
\begin{equation*}
+\Int_{K_S^T}\left(|\nabla u|^2 -|\partial_t u|^2+u^2{\rm e}^{4\pi
u^2}-\Sum_{k=1}^{p-1}\Frac{(4\pi)^ku^{2k+2}}{k!}\right)\;dx\,dt=0,
\end{equation*}
where $B(r)$ is the ball centered at $0$ and of radius $r$ and
$$M_S^T:=\Big\{(t,x)\in\R\times \R^2 \;\mbox{ such that}\;S\leq t\leq T\;\mbox{and}\;|x|=-t\Big\}.$$
According to \eqref{a4} and \eqref{c1}, we infer that
$$\Int_{M_S^T}\left[\left|\partial_t u\frac{x}{|x|}+\nabla u\right|^2+\frac{1}{4\pi}\left({\rm e}^{4\pi u^2}-1-\Sum_{k=2}^p\frac{(4\pi)^ku^{2k}}{k!}\right)\right]\;dx\,dt=0.$$
This implies, using \eqref{c2} and Cauchy-Schwarz inequality, that
\begin{equation}\label{c3}
\Int_{B(-T)}\partial_tu(T)u(T)\;dx-\Int_{B(-S)}\partial_t
u(S)u(S)\;dx
\end{equation}
\begin{equation*}
+\Int_{K_S^T}\left(|\nabla u|^2 -|\partial_t u|^2+u^2{\rm e}^{4\pi
u^2}-\Sum_{k=1}^{p-1}\Frac{(4\pi)^ku^{2k+2}}{k!}\right)\;dx\,dt=0,
\end{equation*}
By virtue of Identities \eqref{a1} and \eqref{a2} and the
conservation law \eqref{energy}, it can be seen that
\begin{equation}\label{c4}
\partial_tu(t)\underset{t\rightarrow 0}\longrightarrow 0\quad\mbox{in}\;L^2(\R^2),
\end{equation}
which ensures by Cauchy-Schwarz inequality that
\begin{equation}\label{c5}
\Int_{B(-T)}\partial_tu(T)u(T)\;dx\rightarrow 0.
\end{equation}
Letting $T\rightarrow 0$ in \eqref{c3}, we deduce from \eqref{c5}
and the fact that $u^2{\rm e}^{4\pi
u^2}-\Sum_{k=1}^{p-1}\Frac{(4\pi)^ku^{2k+2}}{k!}$ is positive
\begin{equation}\label{e1}
-\Int_{B(-S)}\partial_t u(S)u(S)\;dx\leq -\Int_{K_S^0}|\nabla
u|^2\,dx\,dt +\Int_{K_S^0}|\partial_t u|^2\,dx\,dt.
\end{equation}
Multiplying Inequality \eqref{e1} by the positive number
$-\frac{1}{S}$, we deduce that
\begin{equation}\label{e2}
\Int_{B(-S)}\partial_t u(S)\frac{u(S)}{S}\;dx \leq
\frac{1}{S}\Int_{K_S^0}|\nabla u|^2\;dx\;dt
-\frac{1}{S}\Int_{K_S^0}|\partial_t u|^2\;dx\;dt.
\end{equation}
Now, Identity \eqref{c4} leads to
\begin{equation}\label{f1}
\Lim_{S\rightarrow 0^-}\frac{1}{S}\Int_{K_S^0}|\partial_t
u|^2\;dx\;dt=0.
\end{equation}
Moreover, using \eqref{a3}, it is clear that
\begin{equation}\label{f2}
\Lim_{S\rightarrow 0^-} \frac{1}{S}\Int_{K_S^0}|\nabla
u|^2\;dx\;dt=-1.
\end{equation}
Finally, since
\begin{equation*}
    \frac{u(S)}{S}=\frac{1}{S} \int_{0}^{S}\partial_tu(\tau)d\tau,
\end{equation*}
then $(\frac{u(S)}{S})$ is bounded in $L^2(\R^2)$ and hence
\begin{equation}\label{f3}
\Lim_{S\rightarrow 0^-}\Int_{B(-S)}\partial_t
u(S)\frac{u(S)}{S}\;dx=0.
\end{equation}
The identities \eqref{f1}, \eqref{f2} and \eqref{f3} yield a
contradiction in view of \eqref{e2}. This achieves the proof of the
global existence in the critical case.

\subsubsection{Scattering}
Our concern now is to prove that, in the subcritical and critical
cases, the solution of the equation \eqref{exp_p} approaches a
solution of a free wave equation when the time goes to infinity.
Using the fact that \bq\label{scat}
|F_p(u)|\leq|F_2(u)|,\quad\forall p\geq 2, \eq we can apply the
arguments used in \cite{Ibrahim1}. More precisely, in the
subcritical case the key point consists to prove that there exists
an increasing function $C:[0,1[\longrightarrow [0,\infty[$ such
that for any $0\leq E<1$, any global solution $u$ of the Cauchy
problem \eqref{exp_p} with $E_p(u)\leq E$ satisfies
\begin{equation}\label{X}
\|u\|_{X(\R)}\leq C(E),
\end{equation}
where $X(\R)=L^8(\R, L^{16}(\R^2))$. Now, denoting by
$$E^*:=\sup\Big\{0\leq E<1;\; \Sup_{E_p(u)\leq E}\|u\|_{X(\R)}<\infty\Big\},$$
and arguing as in \cite[Lemma 4.1]{Ibrahim1}, we can show that
Inequality \eqref{X} is satisfied if $E_p(u)$ is small, which
implies that $E^*>0$. Now our goal is to prove that $E^*=1$. To do
so, let us proceed by contradiction and assume that $E^*<1$. Then,
for any $E\in]E^*,1[$ and any $n>0$, there exists a global solution
$u$ to \eqref{exp_p} such that $E_p(u)\leq E$ and $\|u\|_{X(\R)}>n$.
By time translation, one can reduce to \bq\label{X1}
\|u\|_{X(]0,\infty[)}>\frac{n}{2}. \eq
Along the same lines as the proof of Proposition 5.1 in \cite{Ibrahim1}, we can show taking advantage of \eqref{scat} that if $E$ is close enough to $E^*$, then $n$ cannot be arbitrarily large which yields a contradiction and ends the proof of the result in the subcritical case.\\
The proof of the scattering in the critical case is done as in
Section 6 in \cite{Ibrahim1} once we observed Inequality
\eqref{scat}. It is based on the notion of concentration radius
$r_\epsilon(t)$ introduced in \cite{Ibrahim1}.

\subsection{Qualitative study }
In this section we shall investigate the feature of solutions of the
two-dimensional nonlinear Klein-Gordon equation \eqref{exp_p} taking
into account the different regimes. As in \cite{Bahouri}, the
approach that we adopt here is the one introduced by P. G\'erard in
\cite{Ge1} which consists in comparing the evolution of oscillations
and concentration effects displayed by sequences of solutions of the
nonlinear Klein-Gordon  equation (\ref{exp_p}) and solutions of the
free linear Klein-Gordon equation.  \beq \label{LKG} \square
v+v=0.\eeq More precisely, let $(\varphi_n, \psi_n)$ be a sequence
of data in $H^1\times L^2$ supported in some fixed ball and
satisfying
\begin{equation}
 \label{weak-conv}
 \varphi_n\rightharpoonup 0\quad\mbox{in}\; H^1,\quad\psi_n\rightharpoonup 0
 \quad\mbox{in}\; L^2,\end{equation}
such that
\begin{equation}
 \label{subcrit}
 E_p^n\leq 1,\quad n\in \N
\end{equation}
where $E_p^n$ stands for the energy of $(\varphi_n, \psi_n)$ given
by $$ E_p^n=\|\psi_n\|_{L^2}^2+\|\nabla\varphi_n\|_{L^2}^2+
\frac{1}{4\pi}\; \Big\|{\rm e}^{4\pi
\varphi_n^2}-1-\Sum_{k=2}^p\frac{(4\pi)^k}{k!}\varphi_n^{2k}\Big\|_{L^1},
$$ and let us consider $(u_n)$ and $(v_n)$ the sequences of finite
energy solutions of (\ref{exp_p}) and (\ref{LKG}) such that $$ (u_n,
\partial_t u_n)(0)=(v_n, \partial_t v_n)(0)=(\varphi _n,\psi_n).
$$ Arguing as in \cite{Ge1}, the notion of linearizability is defined as follows:
\begin{defi}
Let $T$ be a positive time. We shall say that the sequence $(u_n)$
is linearizable on $[0,T]$, if $$
\Sup_{t\in[0,T]}E_c(u_n-v_n,t)\longrightarrow 0\quad\mbox{as}\quad
n\rightarrow\infty, $$ where $E_c(w,t)$ denotes the kinetic energy
defined by:
$$E_c(w,t)=\Int_{{\mathbb{R}}^2}\left[|\partial_t
w|^2+|\nabla_x w|^2+|w|^2\right](t,x)\;dx.
$$
\end{defi}
For any time slab $I\subset\R$, we shall denote
$$
\|v\|_{\mbox{\tiny ST}(I)}:=\sup_{(q,r)\; \mbox{\tiny
admissible}}\;\|v\|_{L^q(I; {\mathrm B}^{1}_{r,2}(\R^2))}\,.
$$
By interpolation argument, this Strichartz norm is equivalent to
$$
\|v\|_{L^\infty(I; H^1(\R^2))}+\|v\|_{L^4(I; {\mathrm
B}^{1}_{8/3,2}(\R^2))}\,.
$$
As $ {\mathrm B}^{1}_{r,2}(\R^2)\hookrightarrow L^p(\R^2)$ for all
$r\leq p<\infty$ (and $r\leq p\leq\infty$ if $r>2$), it follows that
\bq \label{LqLp} \|v\|_{L^q(I; L^p)}\lesssim\|v\|_{\mbox{\tiny
ST}(I)},\quad \frac{1}{q}+\frac{2}{p}\leq 1\,. \eq
 As in \cite{Bahouri}, in the subcritical case, i.e $\Limsup_{n\rightarrow\infty}\;E_p^n<1$,
the nonlinearity does not induce any effect on the behavior of the
solutions. But, in the critical case i.e
$\Limsup_{n\rightarrow\infty}\;E_p^n=1$, it turns out that a
nonlinear effect can be produced. More precisely, we have the
following result:
\begin{thm}\label{both cases}
Let $T$ a strictly positive time. Then
\begin{enumerate}
  \item If $\underset{n\rightarrow
\infty}{\limsup}\,E_p^n<1$, the sequence $(u_n)$ is linearizable on
$[0,T]$.
  \item If
$\underset{n\rightarrow \infty}{\limsup}\, E_p^n=1$, the sequence
$(u_n)$ is linearizable on $[0,T]$ provided that the sequence
$(v_n)$ satisfies \beq
\label{crit-cond}\limsup_{n\to\infty}\;\|v_n\|_{L^\infty([0,T];
{L^{\Phi_p}})}<\frac{1}{\sqrt{4\pi}}\cdot \eeq

\end{enumerate}
  \end{thm}
\begin{proof}
The proof of Theorem \ref{both cases} is similar to the one of
Theorems 3.3 and 3.5 in \cite{Bahouri}. Denoting by~$w_n= u_n- v_n$,
it is clear that $w_n$ is the solution of the nonlinear wave
equation $$  \square w_n + w_n =
-F_p(u_n) $$ with null Cauchy data.\\
Under energy estimate, we obtain $$\|w_n\|_T \lesssim
\|F_p(u_n)\|_{L^1([0,T],L^2(\R^2))},$$ where~$\|w_n\|^2_T \eqdefa
\sup_{t\in [0,T]} E_c(w_n,t)$. Therefore, it suffices to prove in
the subcritical and critical cases that
 \begin{equation}\label{Fp} \|F_p(u_n)\|_{L^1([0,T],L^2(\R^2))}\longrightarrow
0\quad\mbox{as}\quad n\rightarrow\infty.
\end{equation}
Let us begin by the subcritical case. Our goal is to prove that the
nonlinear term does not affect the behavior of the solutions. By
hypothesis, there exists some nonnegative real $\rho$ such that
$\Limsup_{n\rightarrow\infty}E_p^n=1-\rho$. The main point for the
proof is based on the following lemma, the proof of which is similar
to the proof of Lemma 3.16 in \cite{Bahouri} once we observed
Inequality \eqref{scat}.
\begin{lem}
\label{subc-est} For every $T>0$ and $E^{0}_{p}<1$, there exists a
constant $C(T,E^{0}_{p})$, such that every solution $u$ of the
nonlinear Klein-Gordon  equation \eqref{exp_p} of energy $E_p(u)\leq
E^{0}_{p}$, satisfies \beq \label{subc-stri} \|u\|_{L^4([0,T];
{\cC}^{1/4})}\leq C(T,E^{0}_{p}).
 \eeq
\end{lem}
Now to establish \eqref{Fp}, it suffices to prove that the sequence
$(F_p(u_n))$ is bounded
in~$L^{1+\epsilon}([0,T],L^{2+\epsilon}(\R^2))$ for some nonnegative
~$\epsilon $ and converges
to~$0$ in measure in~$[0,T]\times \R^2$. This can be done exactly as in \cite{Bahouri} using the fact that $|F_p(u_n)|\leq|F_1(u_n)|$.\\

Let us now prove \eqref{Fp} in the critical case. For that purpose,
let $T>0$ and assume that \bq \label{crit-assum}
L:=\limsup_{n\to\infty}\;\|v_n\|_{L^\infty([0,T];{L^{\Phi_p}})}<\frac{1}{\sqrt{4\pi}}\cdot
 \eq
 Applying Taylor's formula, we obtain
$$
F_p(u_n)=F_p(v_n+w_n)=F_p(v_n)+F_p'(v_n)\,w_n+\frac{1}{2}\;F_p''(v_n+\theta_n\,w_n)\,w_n^2,
$$
for some $0\leq \theta_n\leq 1$. Strichartz estimates \eqref{e6}
yields
$$\|w_n\|_{\mbox{\tiny ST}([0,T])}\lesssim I_n+J_n+K_n,$$
where \begin{eqnarray*}
I_n&=&\|F_p(v_n)\|_{L^1([0,T]; L^2(\R^2))},\\
J_n&=&\|F_p'(v_n)\,w_n\|_{L^1([0,T]; L^2(\R^2))},\quad\mbox{and}\\
K_n&=&\|F_p''(v_n+\theta_n\,w_n)\,w_n^2\|_{L^1([0,T]; L^2(\R^2))}.
\end{eqnarray*}
As in \cite{Bahouri}, we have
\begin{eqnarray*}
I_n&\underset{n\rightarrow\infty}\longrightarrow& 0\quad\mbox{and}\\
J_n&\leq& \varepsilon_n \|w_n\|_{ST([0,T])},
\end{eqnarray*}
where $\varepsilon_n\rightarrow 0$. Besides, provided that \bq
\label{claim31} \limsup_{n\to\infty}\,\|w_n\|_{L^\infty([0,T];
H^1)}\leq \frac{1-L\,\sqrt{4\pi}}{2}, \eq we get
$$K_n \leq \varepsilon_n \|w_n\|_{ST([0,T])}^2,\quad \varepsilon_n\rightarrow 0.$$
Since $\|w_n\|_{ST([0,T])}\lesssim I_n+\varepsilon_n
\|w_n\|_{ST([0,T])}^2$, wet obtain by bootstrap argument
$$\|w_n\|_{ST([0,T])}\lesssim\varepsilon_n,$$
which ends the proof of the result.
\end{proof}

\section{Appendix: Proof of Proposition \ref{Mos3}}
 The proof uses in a crucial way the rearrangement of functions (for a
complete presentation and more details, we refer the reader to
\cite{M}).
 By virtue of  density arguments and the fact that for any function $f\in H^1(\R^2)$ and $f^*$ the rearrangement of f, we have
\beqn
\|\nabla f\|_{L^2}&\geq &\|\nabla f^*\|_{L^2},\\
\|f\|_{L^p}&=&\|f^*\|_{L^p},\\
\|f\|_{L^{\phi_p}}&=&\|f^*\|_{L^{\phi_p}}\,, \eeqn one can reduce to
the case of a nonnegative radially symmetric and non-increasing
function  $u$ belonging to  ${\cD}(\R^2)$.  With this choice, let us
introduce
 the function $$w(t)=(4\pi)^{\frac{1}{2}} u(|x|), \quad \mbox{where} \quad|x|={\rm e}^{-\frac{t}{2}}.$$ It is then obvious  that the functions  $ w(t) $ and
 $ w'(t)$ are nonnegative and satisfy
\begin{eqnarray*}
  \int_{\mathbb{R}^2}|\nabla u(x)|^2 \,dx &=& \int_{-\infty}^{+\infty}|{w}'(t)|^2 \,dt, \\ \int_{\mathbb{R}^2}|u(x)|^{2p}\,dx &=& \frac{1}{4^p \, \pi^{p-1} }
  \int_{-\infty}^{+\infty} |w(t)|^{2p}~{\rm
  e}^{-t}\,dt,
   \\
 \dint_{\mathbb{R}^2} \left({\rm e}^{\alpha |u(x)|^2}-\sum_{k=0}^{p-1}\frac{\alpha^k |u(x)|^{2k}}{k!}\right)dx &=& \pi \dint_{-\infty}^{+\infty} \left({\rm e}^{\frac{\alpha}{4\pi}|w(t)|^2}-\sum_{k=0}^{p-1} \frac{\alpha^k |w(t)|^{2k}}{(4\pi)^k k!}\right) {\rm e}^{-t}\,dt.
\end{eqnarray*}  So we are reduced to prove that for any $\beta\in  [0,1[$, there  exists $C_{\beta}\geq 0$ so that \begin{equation*}\label{res}
\dint_{-\infty}^{+\infty} \left({\rm e}^{\beta
|w(t)|^2}-\sum_{k=0}^{p-1} \frac{\beta^k |w(t)|^{2k}}{k!}\right)
{\rm e}^{-t} dt \leq C({\beta,p})
\dint_{-\infty}^{+\infty}|w(t)|^{2p} {\rm
e}^{-t}\,dt,\quad \forall \,\beta\in  [0,1[,
\end{equation*}
when $\displaystyle \int_{-\infty}^{+\infty}|{w}'(t)|^2 dt \leq1.$
For that purpose, let us  set $$T_{0}=\sup \bigg\lbrace{t \in
\mathbb{R},~w(t)\leq 1\bigg\rbrace}.$$ The existence of a real
number $t_0$ such that $w(t_0)=0$ ensures that the set
$\bigg\lbrace{t \in \mathbb{R},~w(t)\leq 1\bigg\rbrace}$ is non
empty. Then
$$T_0\in]-\infty,+\infty].$$
Knowing that $w$ is nonnegative  and increasing function, we deduce
that
$$ w:   ] -\infty,T_0 ] \longrightarrow  [0,1]. $$  Therefore, observing that  $\displaystyle {\rm e}^{s}-\sum_{k=0}^{p-1} \frac{s^k}{k!} \leq c_p \,s^p \,{\rm e}^{s}$ for any nonnegative real  $s $, we obtain
\begin{equation*}
\dint_{-\infty}^{T_0} \left({\rm e}^{\beta
|w(t)|^2}-\sum_{k=0}^{p-1} \frac{\beta^k |w(t)|^{2k}}{k!}\right)
{\rm e}^{-t} dt \leq c_p\,\beta^p \,{\rm e}^{\beta}
\dint_{-\infty}^{T_0} |w(t)|^{2p} {\rm e}^{-t} dt.
\end{equation*}
To estimate the integral on $[T_0,+\infty[$, let us first notice
that in view of the definition of $T_0$, we have for all $t\geq T_0$
\begin{eqnarray*}
  w(t) &=& w(T_0)+\int_{T_0}^t{w}'(\tau)d\tau \\
  &\leq& w(T_0)+(t-T_0)^{\frac{1}{2}}\left(\int_{T_0}^{+\infty}{w'}(\tau)^2 d\tau\right)^{\frac{1}{2}}\\
  &\leq& 1+(t-T_0)^{\frac{1}{2}}.\end{eqnarray*}
Thus, using the fact  that for any $\varepsilon > 0 $ and any $s\geq
0$, we have
$$(1+s^{\frac{1}{2}})^2 \leq (1+ \varepsilon) s + 1+\frac{1}{\varepsilon}=(1+ \varepsilon) s+C_{\varepsilon},$$
we infer that for for any $\varepsilon > 0 $ and all $t\geq T_0$
\begin{equation}\label{estus}|w(t)|^2 \leq  (1+\varepsilon)(t-T_0)+ C_{\varepsilon}. \end{equation}
Now $\beta$ being fixed in  $[0,1 [$, let us choose $\varepsilon>0$
so that $\beta(1+\varepsilon)<1$. Then by virtue of \eqref{estus}
\begin{eqnarray*}
        \int_{T_0}^{+\infty} \Big({\rm e}^{\beta |w(t)|^2}-\sum_{k=0}^{p-1}\frac{\beta^k  |w(t)|^{2k}}{k!}\Big) {\rm e}^{-t}\,dt&\leq & \int_{T_0}^{+\infty}{\rm e}^{\beta |w(t)|^2} {\rm e}^{-t}\,dt \\
        &\leq&\frac{{\rm e}^{\beta  C_{\varepsilon}-T_0}}{1-\beta(1+\varepsilon)}\cdot
      \end{eqnarray*}
     But
$${\rm e}^{-T_0}=\int_{T_0}^{+\infty} {\rm e}^{-t}\,dt \leq \int_{T_0}^{+\infty}  |w(t) |^{2p} \, {\rm e}^{-t}\,dt,$$
which gives rise to
\begin{equation*}
    \int_{T_0}^{+\infty} \Big({\rm e}^{\beta |w(t)|^2}-\sum_{k=0}^{p-1}\frac{\beta^k |w(t)|^{2k}}{k!}\Big){\rm e}^{-t} dt \leq \frac{{\rm e}^{\beta C_{\varepsilon}}}{1-\beta(1+\varepsilon)} \int_{T_0}^{\infty}  |w(t) |^{2p} {\rm e}^{-t}\,dt.\end{equation*}
Choosing $C({\beta,p})= \max \Big(c_p{\rm e}^{\beta} \beta^p
,\displaystyle \frac{{\rm e}^{\beta
C_{\varepsilon}}}{1-\beta(1+\varepsilon)}\Big)$ ends the proof of
the proposition.


\end{document}